\documentclass[a4paper]{article}
\usepackage[affil-it]{authblk}
\usepackage{amsmath}
\usepackage{amsthm}
\usepackage{amssymb}
\usepackage{graphicx}
\usepackage{comment}
\usepackage{color}

\newtheorem*{theorem}{\bf Theorem}
\newtheorem*{lemma}{\bf Lemma}

\providecommand{\keywords}[1]{\textbf{\textit{Keywords and phrases---}} #1}

\begin{document}
\title{Unifying Dynamical and Structural \\ Stability of Equilibriums}
\date{\vspace{-5ex}}

\author{Jean-Fran\c{c}ois Arnoldi\thanks{Electronic address: \texttt{jean-francois.arnoldi@sete.cnrs.fr}; Corresponding author} $\;$and Bart Haegeman}
\affil{Center for Biodiversity Theory and Modeling, Station d'\'Ecologie Th\'eorique et Exp\'erimentale, CNRS and Universit\'e Paul Sabatier,  09200 Moulis, France}

\maketitle

\begin{abstract}
We exhibit a fundamental relationship between measures of dynamical and structural stability of equilibriums, arising from real dynamical systems.  We show that dynamical stability, quantified via systems local response to external perturbations, coincides with the minimal internal perturbation able to destabilize the equilibrium.  
First, by reformulating a result of control theory, we explain that harmonic external perturbations reflect the spectral sensitivity of the Jacobian matrix at the equilibrium, with respect to constant changes of its coefficients.  However, for this equivalence to hold, imaginary changes of the Jacobian's coefficients have to be allowed.  The connection with dynamical stability is thus lost for real dynamical systems.  
We show that this issue can be avoided, thus recovering the fundamental link between dynamical and structural stability, by considering stochastic noise as external and internal perturbations.  More precisely, we demonstrate that a system's local response to white-noise perturbations directly reflects the intensity of internal white noise that it can accommodate before asymptotic mean-square stability of the equilibrium is lost. 
\end{abstract}
\maketitle

\section{Prologue} 

Understanding the stability of dynamical systems is fundamental in numerous applications, from classical mechanics, fluid dynamics, engineering, to biology \cite{schmid_stability_2012, chandrasekhar_hydrodynamic_1970, penzien_dynamics_1975, machowski_power_2011, steinfeld_chemical_1999, may_stability_1973}.  Stability refers to the ability of a subset in phase space to attract trajectories from its vicinity. In otherwords, a given state is dynamically stable if trajectories remains close to that state despite small perturbations. Such sets are called attractors, the most basic kind being equilibriums. Regardless of their conceptual simplicity, they commonly appear in a large variety of models.  For instance, their study in ecology is fundamental to understand the mechanisms  that stabilize ecosystems and support their staggering diversity \cite{may_stability_1973}.  In fluid mechanics, the laminar state can be seen as an equilibrium, and the transition to turbulence as a loss of local stability \cite{schmid_stability_2012}.  In the context of dynamics of nodes of electric power grids, the equilibrium is a desired state where the generator operates in synchrony with the grid \cite{machowski_power_2011}.

For equilibriums, the necessary and sufficient condition for stability is derived in terms of the spectrum of the associated Jacobian matrix.  If the real parts of its eigenvalues are negative then the equilibrium is locally asymptotically stable.  This local stability analysis has been fruitfully applied across disciplines.

\noindent On the other hand, structural stability relates to the robustness of the qualitative dynamical picture with respect to small changes in the system structure \cite{thom_structural_1989, katok_introduction_1997}.  This notion is particularly important when the system is a simplified model of a more complicated real-world system, which is virtually always the case in applications.  For the model to inform on the real-world system it must be robust with respect to small perturbations, uncertainties and so forth.  There are broad classes of models that are known to be structurally stable, the most basic ones being hyperbolic linear systems, a result that justifies the study of linearized models in the vicinity of equilibriums. The two stability notions described above are qualitative, yet it is often important to quantify stability, either dynamical or structural, in order to compare different models or to assess the effect of parameters on stability.  Qualitative notions answer the question \emph{is a particular state (or model) stable?}, while quantitative measures answer \emph{how stable is this state (or model)?}  Dynamical stability is typically quantified via the system response to pulse-perturbations, that is, instantaneous displacements in phase space, but other perturbations are also important. For instance, periodic forcing can reveal resonances.  Although far less common in the literature \cite{rohr_structural_2014, grilli_geometry_2015}, measures of structural stability are by no means less relevant.  They quantify the stability of the system itself, that is, the intensity of structural perturbations that it can accommodate before its behavior qualitatively changes.

\noindent In this article we focus on equilibriums, arising from real dynamical systems.  We introduce natural measures of dynamical stability, quantifying systems local response to persistent forcing.   We compare them to natural measures of structural stability, quantifying the smallest change in the local dynamical structure leading to destabilization.  We show that these measures coincide with one another, so that dynamical response to external perturbations can inform on systems sensitivity to structural changes. 

\noindent In the first section we revisit a result of control theory, showing that harmonic external perturbations reflect the spectral sensitivity of the Jacobian matrix at the equilibrium, with respect to constant changes of its coefficients.  On an elementary example, we illustrate a caveat of this approach, showing that this relationship does not hold in the context of real dynamical systems.  In the following section, in which our main result is stated, we demonstrate that the fundamental link between dynamical and structural stability can be recovered by considering stochastic noise as external and internal perturbations.

\section{Dynamical and structural stability: Harmonic perturbations} 

The standard procedure to assess stability of an equilibrium consists in linearizing the vector field in its vicinity, effectively reducing the local dynamics to a linear system of the form $\dot{\boldsymbol{x}}=A\boldsymbol{x}$, where $A$ is the Jacobian matrix evaluated at the equilibrium.  Defining the spectral abscissa of $A$ as 
\begin{equation}
 \alpha\left(A\right) =
 \sup\left\{ \Re(\lambda)\;|\;\lambda\in\mbox{spect}\left(A\right)\right\},
\end{equation}
we say that a matrix with negative spectral abscissa is stable (i.e., the associated equilibrium is stable) and unstable otherwise. 

A straightforward way to quantify the dynamical stability of a stable equilibrium is to analyze the system's local response to harmonic forcing. This amounts to solving
\[
 \dot{\boldsymbol{x}}=A\boldsymbol{x}+\Re(e^{i\omega t}\boldsymbol{u}),
\]
where $\omega\in\mathbb{R}$ is the frequency of a real rotating perturbation.  The stationary response is $\Re\big(e^{i\omega t}\boldsymbol{w}\big)$ with $\boldsymbol{w} = \left(i\omega-A\right)^{-1}\boldsymbol{u}$.  The norm of the complex vector $\boldsymbol{w}$ is the mean amplitude of the induced oscillations.  The spectral norm of the matrix $\left(i\omega-A\right)^{-1}$ gives the strongest system response to harmonic forcing of frequency $\omega$.  To define a measure of stability, we take the inverse of the largest system amplification of harmonic forcing. This translates as
\begin{equation}
 \mathcal{S}_{_{\mathrm{DYN}}}^{\mathfrak{h}}(A)
 = 1/\sup_{\omega\in\mathbb{R}}\big\Vert (i\omega-A)^{-1}\big\Vert.
 \label{eq:variability-1}
\end{equation}
The number $\omega$ realizing the maximum is called the resonant frequency.  It can be shown \cite{arnoldi_resilience_2016} that $\mathcal{S}_{_{\mathrm{DYN}}}^{\mathfrak{h}}$ relates to the maximal power gain over wide-sense stationary signals, indicating that, although defined with respect to a specific class of forcing, it is a general indicator of the ability of an equilibrium to absorb external perturbations. 

Let us now turn to the problem of quantifying structural stability.  For equilibriums we may consider how close the Jacobian matrix $A$ is from being unstable, that is, the minimal constant change in its coefficients that can push its dominant eigenvalue into the instability region $\{ z\in \mathbb{C} | \Re(z) \geq 0 \}$ of the complex plane.   Measuring the distance to instability as the spectral norm of the smallest matrix $P$ rendering $A+P$ unstable, this yields
\begin{equation}
 \mathcal{S}_{_{\mathrm{STR}}}^c(A)
 = \inf\left\{ \Vert P\Vert\;|\;\alpha(A+P)>0\right\}
 \leq |\alpha(A)| \}. \label{eq:r_C}
\end{equation}
This definition of structural stability is also known as the stability radius of $A$ \cite{hinrichsen_stability_1986}.  The inequality in (\ref{eq:r_C}) comes from the fact that the perturbation $P=-\alpha(A)\,\mathbb{I}$ is always sufficient to destabilize $A$.  In fact, it is the most efficient way to destabilize $A$ when $A$ is normal (i.e. has orthonormal eigenvectors), in which case the inequality is an equality \cite{trefethen_spectra_2005}. Note that the spectral absissa $|\alpha(A)|$ is the Euclidian distance in the complex plane to instability. Hence the two distances $-$stability radius and spectral absissa$-$ coincide when the jacobian matrix is normal.

\noindent There is a strong link between $\mathcal{S}_{_{\mathrm{STR}}}^c$ and the dynamical measure $\mathcal{S}_{_{\mathrm{DYN}}}^{\mathfrak{h}}$ introduced in (\ref{eq:variability-1}).  To reveal this link suppose that for some stable matrix $A$, $\mathcal{S}_{_{\mathrm{DYN}}}^{\mathfrak{h}}(A) = 1/v$ where $v>0$ is the strongest response associated to the resonance $\omega$.  Pick two normalized vectors: $\boldsymbol{u}$, spanning the direction of perturbation, and $\boldsymbol{w}$, spanning the direction of response, both associated to the resonance $\omega$.  We have that
\[
 (i\omega-A)^{-1}\boldsymbol{u} = v\boldsymbol{w}
 \ \Leftrightarrow\ 
 A\boldsymbol{w}+v^{-1}\boldsymbol{u} = i\omega\boldsymbol{w}.
\]
We can construct a destabilizing matrix from the vectors $\boldsymbol{u}$ and $\boldsymbol{w}$.  This is done by choosing $P = 
v^{-1}\boldsymbol{u}\boldsymbol{w}^{*}$, so that $\left\Vert P \right\Vert = v^{-1}$, $P\boldsymbol{w} = v^{-1}\boldsymbol{u}$ and 
\[
 \left(A+P\right)\boldsymbol{w} = i\omega\boldsymbol{w}
 \ \Rightarrow\ 
 \alpha\left(A+P\right) = 0.
\]
Hence, $P$ destabilizes $A$, meaning that $\mathcal{S}_{_{\mathrm{STR}}}^c(A) \leq \left\Vert P\right\Vert 
= v^{-1}= \mathcal{S}_{_{\mathrm{DYN}}}^{\mathfrak{h}}(A)$. 

\noindent Conversely, suppose that $\mathcal{S}_{_{\mathrm{STR}}}^c(A) = p$.  There exists a matrix $P$ with $\left\Vert P\right\Vert = p$ such that $A+P$ is unstable: for some $\omega$ and normalized vector $\boldsymbol{w}$,
\[
 \left(A+P\right)\boldsymbol{w} = i\omega\boldsymbol{w}
 \ \Leftrightarrow\ 
 \boldsymbol{w} = \left(i\omega-A\right)^{-1}\boldsymbol{u},
\]
with $\boldsymbol{u} = P\boldsymbol{w}$.  Since $\left\Vert\boldsymbol{u}\right\Vert \leq p$ we deduce that $\Vert(i\omega-A)^{-1}\Vert \geq p^{-1}$.  Hence,
\begin{equation}
 \mathcal{S}_{_{\mathrm{DYN}}}^{\mathfrak{h}}
 = \mathcal{S}_{_{\mathrm{STR}}}^c
 \label{eq:formula_rC}
\end{equation}
giving from (\ref{eq:variability-1}) a computable expression for structural stability.  Equation~(\ref{eq:formula_rC}) corresponds to a well known result in control theory \cite{hinrichsen_stability_1986}, which we interpret here in terms of dynamical and structural stability of equilibriums. 

\noindent There is however a caveat. The quantitative measure of structural stability we have considered allows for complex matrix perturbations, that almost never make sense in applications.  In fact, computing the corresponding real structural stability, which we denote as $\mathcal{S}_{_{\mathrm{STR}}}^{\Re(c)}$, involves a complicated global optimization problem \cite{qiu_formula_1995}.  In general dynamical stability can be much smaller than its real structural counterpart.  This issue is particularly apparent in the following elementary example.  Consider the sequence of Jacobian matrices,
\begin{equation}
 A = \left(\begin{array}{cc}
 -1 & M^{2} \\ -1 & -1
 \end{array}\right)
 \quad \text{with }\ M=1,2,\ldots,
 \label{eq:example}
\end{equation}
whose eigenvalues are $-1\pm iM$, so that $\alpha(A)=-1$.  The associated equilibriums are stable for all values of $M$.  The strongest response to harmonic forcing grows with $M$. Also, complex perturbations have an effect of order $M$ on the real part of the spectrum, so that perturbations of order $M^{-1}$ can destabilize the matrix. 

\noindent This is not true for real perturbations as,
\[
 1 = \mathcal{S}_{_{\mathrm{STR}}}^{\Re(c)}(A)
   > \mathcal{S}_{_{\mathrm{STR}}}^c(A)
   = \mathcal{S}_{_{\mathrm{DYN}}}^{\mathfrak{h}}(A)
   \rightarrow_{M\rightarrow\infty} 0.
\]
Real structural stability can thus be completely disconnected from its dynamical counterpart.

\section{Dynamical and structural stability: White-noise perturbations} 

Let us now transpose the relationship between dynamical and structural stability to white-noise forcing, often used to model the effect of erratic external perturbations \cite{arnoldi_resilience_2016, buckwar_asymptotic_2014}.  White noise acts locally as
\begin{equation}
 d\boldsymbol{X}_{t} = A\boldsymbol{X}_{t}dt+Td\boldsymbol{W}_{t},
\end{equation}
where $\boldsymbol{W}_{t}$ is a vector of independent Wiener processes, representing various external factors acting on the system, with the matrix $T$ describing how these factors affect system variables.  The first moments $\boldsymbol{\mu}_t = \mathbb{E}\boldsymbol{X}_t$ evolve as $\dot{\boldsymbol{\mu}}_t = A\boldsymbol{\mu}_t$ and converge to zero if $A$ is stable.  The second moments, represented as covariance matrices $C_t = \mathbb{E}\boldsymbol{X}_t\boldsymbol{X}_t^{\top}$, follow \cite{arnold_stochastic_1976,kampen_stochastic_1997}
\begin{equation}
 \dot{C}_{t} = \widehat{A}C_{t} + \Sigma,
 \label{eq:lifted-dynamics}
\end{equation}
with $\widehat{A}C = AC + CA^{\top}$, called hereafter the lifted operator, and $\Sigma = TT^{\top}$, a positive semi-definite matrix, encoding the effective correlations of the noise.  If $A$ is stable, any initial covariance matrix converges to
\[
 C_{*} = -\widehat{A}^{-1}\Sigma,
\]
the unique attractor of (\ref{eq:lifted-dynamics}). 

\noindent In analogy with the measure $\mathcal{S}_{_{\mathrm{DYN}}}^{\mathfrak{h}}$ constructed via the largest local response to normalized harmonic perturbations, we define a measure of dynamical stability  by taking the inverse of the strongest system response over normalized white-noise perturbations. This leads us to
\begin{equation}
 \mathcal{S}_{\mathrm{_{DYN}}}^{\mathrm{w}}(A)
 = 1 / \sup_{\Sigma \geq 0,\,\Vert\Sigma\Vert_\text{F}=1}
 \big\Vert \widehat{A}^{-1}\Sigma \big\Vert_{\mathrm{F}},
 \label{eq:stochastic_variability}
\end{equation}
where the supremum is taken over covariance matrices of the real external noise.  The use of the Frobenius norm, $\Vert\Sigma\Vert_{\mathrm{F}} = \text{Tr}\big(\Sigma^{\top}\Sigma\big)^{\frac{1}{2}}$, to normalize the correlation matrices allows us to see them as vectors endowed with the usual scalar product and Euclidean norm.  Because $-\widehat{A}^{-1}$ is a completely positive map, the matrix $\Sigma$ realizing the norm $\Vert\widehat{A}^{-1}\Sigma\Vert_{\mathrm{F}} = \Vert\widehat{A}^{-1}\Vert$ is a positive semi-definite matrix \cite{watrous_notes_2005}.  We thus get that 
\begin{equation}
 \mathcal{S}_{\mathrm{_{DYN}}}^{\mathrm{w}}(A)
 = 1/\big\Vert \widehat{A}^{-1} \big\Vert.
 \label{eq:V_S}
\end{equation}
Note that the lifted operator $\widehat{A}$ can be expressed as a larger matrix $A\otimes\mathbb{I} + \mathbb{I}\otimes A$, giving a simple way to compute $\mathcal{S}_{\mathrm{_{DYN}}}^{\mathrm{w}}$. 

\noindent So far, as in control theory, we considered constant changes in the Jacobian matrix to quantify structural stability.  We now embark on a different path, assuming that the coefficients of the Jacobian matrix fluctuate.  In time-series analysis such variations are called process errors, while the ones previously considered would correspond to observation errors \cite{chatfield_analysis_1989}.  To model the effect of internal perturbations, we pick a set of real matrices $P_{k}$ and independent Wiener processes $W_{t}^{k}$, and consider
\begin{equation}
 d\boldsymbol{X}_{t} = \Big( Adt+\sum_k P_{k}\,dW_{t}^{k} \Big) \boldsymbol{X}_{t}.
 \label{eq:homogeneous_SDE}
\end{equation}
The matrices $P_{k}$ describe the fluctuations of the matrix entries $A_{ij}$ and their correlations.  For example, independent fluctuations of variance $\sigma^2$ of all entries $A_{ij}$ would correspond to $P_k=\sigma \boldsymbol{e}_i\boldsymbol{e}_j^{\top}$ where $\{\boldsymbol{e}_i\}$ stands for the standard orthonormal basis.  Note that the representation of multiplicative noise in (\ref{eq:homogeneous_SDE}) corresponds to It\^o's interpretation of stochasticity \cite{kampen_stochastic_1997}.  We will discuss this point further in the 
Epilogue.
The first moments follow the same dynamics as before, $\dot{\boldsymbol{\mu}}_t = A\boldsymbol{\mu}_t$, and converge to zero if $A$ is stable.  As before, equation~(\ref{eq:homogeneous_SDE}) can be lifted to act on covariance matrices as \cite{arnold_stochastic_1976,kampen_stochastic_1997}
\begin{equation}
 \dot{C_{t}} = \big(\widehat{A}+\mathcal{P}\big) C_{t},
 \label{eq:lifted_hom_sde}
\end{equation}
with $\mathcal{P}(C) = \sum_{k}P_{k} C P_{k}^{\top}$.  We measure the intensity of the internal perturbation by the spectral norm $\Vert\mathcal{P}\Vert$.  In the case of independent fluctuations of all entries of $A$, $\Vert\mathcal{P}\Vert = n^2\sigma^2$, with $n$ the system dimension.  We can then define stochastic structural stability as the minimal internal perturbation intensity able to destabilize the second moments of (\ref{eq:homogeneous_SDE}),
\begin{equation}
 \mathcal{S}_{\mathrm{_{STR}}}^{\mathrm{w}}(A)
 = \inf\{ \left\Vert \mathcal{P} \right\Vert 
 \;|\; \alpha\big(\widehat{A}+\mathcal{P}\big) > 0 \},
 \label{eq:r_S}
\end{equation}
where the infimum is over perturbations $\mathcal{P}$ constructed from an arbitrary sequence of real matrices $P_{k}$. Note that an arbitrary small internal perturbation acting as the multiplicative noise in (\ref{eq:homogeneous_SDE}) can destabilize moments of high enough order. This is apparent in the one-dimension case $dX=(-adt+pdW_t)X$ whose nth-order moments $\mu_n=\mathbb{E}X^n$ satisfy $\dot{\mu}_n=n(-a+p^2(n-1)/2)\mu_n$. Hence, as soon as $p\neq0$,  moments of order $n \geq 2a/p^2+1$ diverge as time flows forward. However, as long as the second moments are bounded, Chebyshev's inequality, that plays a pivotal role in the theory of persistence \cite{schreiber_persistence_2012}, provides a control on the probability of excursions away from equilibrium. In our context, if $C_0$ is the covariance matrix of initial conditions, Chebyshev's inequality reads
\begin{equation}
\mathbb{P}(|X_i(t)|\geq \delta, \; \mathrm{for \; some}\;1\leq i \leq n)   \leq \frac{1}{\delta^2}\left\Vert e^{t(\widehat{A}+\mathcal{P})}C_0 \right\Vert_\mathrm{Tr} 
\end{equation}
which, by virtue of (\ref{eq:lifted_hom_sde}) and equivalence of Trace and Frobenius norms, goes to zero as time flows foward if $\widehat{A}+\mathcal{P}$ is stable $-$regardless of possible divergent moments of higher order. In other words, loosing stability in the sense of (\ref{eq:r_S}) implies loosing control on the probability of excursions from the equilibrium. The importance of this kind of probabilistic stability, called \emph{mean-square asymptotic stability}, is discussed in  \cite{buckwar_asymptotic_2014}, with examples from ecology, turbulent fluid mechanics, and system control. 

With these definitions in hand, and probabilistic interpretation in mind, we can now state our main result, in the form of the following:

\begin{theorem}
For stable equilibriums of real dynamical systems, local measures of structural and dynamical stability coincide, in the sense that
\begin{equation}
 \mathcal{S}_{\mathrm{_{STR}}}^{\mathrm{w}}
 = \mathcal{S}_{\mathrm{_{DYN}}}^{\mathrm{w}}.
 \label{eq:r_S_vs_V_S}
\end{equation}
Here dynamical stability is quantified as the inverse of the maximal system response to external white-noise perturbation, and structural stability as the minimal internal white-noise perturbation needed for the equilibrium to loose mean-square asymptotic stability.  
The latter relates to $\mathcal{S}_{\mathrm{_{STR}}}^{c}$ (resp.  $\mathcal{S}_{\mathrm{_{STR}}}^{\Re(c)}$) $-$the minimal complex (resp. real) constant perturbation of the Jacobian matrix at the equilibrium able to render the latter unstable$-$ via the following inequality 
\begin{equation}
 \mathcal{S}_{\mathrm{_{STR}}}^{\mathrm{w}}
 \leq 2 \mathcal{S}_{\mathrm{_{STR}}}^c 
 \leq 2 \mathcal{S}_{\mathrm{_{STR}}}^{\Re(c)} 
 \leq  2|\alpha| \label{eq:main-result}
\end{equation}
whith $\alpha$ the spectral abscissa of the Jacobian matrix at the equilibrium. (\ref{eq:main-result}) is an equality when the Jacobian at the equilibrium is a normal matrix. 
\end{theorem}

\noindent The example of Jacobian matrices~(\ref{eq:example}) is revisited in figure~\ref{fig:noise-destab}.  We see that the low stability with respect to constant imaginary perturbations detected by $\mathcal{S}_{\mathrm{_{STR}}}^c$ is also present when considering real stochastic fluctuations in the matrix coefficients, as predicted by $\mathcal{S}_{\mathrm{_{STR}}}^{\mathrm{w}}$.  Inequality~(\ref{eq:main-result}) is illustrated in figure~\ref{fig:test_conjecture} showing that, although associated to real perturbations,  $\mathcal{S}_{\mathrm{_{STR}}}^{\mathrm{w}}$ can sometimes be much smaller than its deterministic and complex counter-part  $\mathcal{S}_{\mathrm{_{STR}}}^c$.

\begin{figure}
\begin{center}
\includegraphics[width=0.9\columnwidth]{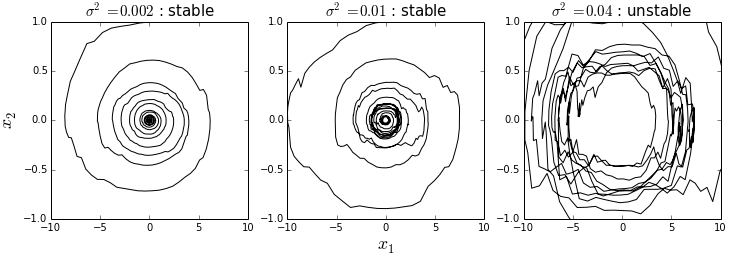}
\end{center}
\newsavebox{\smlmata}
\newsavebox{\smlmatb}
\savebox{\smlmata}{$A = \big( \begin{smallmatrix} -1 & 100 \\ -1 & -1 \end{smallmatrix} \big)$}
\savebox{\smlmatb}{$P = \big( \begin{smallmatrix} -0.07 & -0.27 \\ -0.92 & 0.37 \end{smallmatrix} \big)$}
\mbox{} \\[-24pt] \caption{
\label{fig:noise-destab}
\textbf{Stochastic destabilization by internal white-noise perturbation.}  The Jacobian matrix is \usebox{\smlmata}.  We have that $1/\Vert\widehat{A}^{-1}\Vert \simeq 0.04$, so that according to (\ref{eq:r_S_vs_V_S}) fluctuations $\mathcal{P}$ with intensity $\Vert\mathcal{P}\Vert \geq 0.04$ affecting the matrix $A$ can destabilize the equilibrium.  We show a realization of the process $d\boldsymbol{X}_{t} = (Adt+\sigma PdW_t)\boldsymbol{X}_{t}$ with \usebox{\smlmatb} and $\Vert P\Vert =1$. In the rightmost panel the variance $\sigma^{2} =0.04$ is large enough to show premises of destabilization.  Recall that for this matrix the real stability radius was independent of $M$, with  $\mathcal{S}_{\mathrm{_{STR}}}^{\Re(c)}(A) = 1$.
}
\end{figure}

\begin{figure}
\begin{center}
\includegraphics[width=0.6\columnwidth]{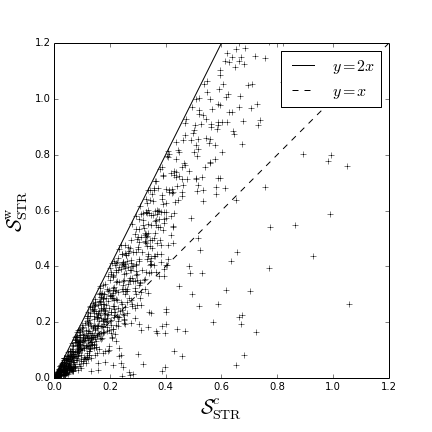} 
\end{center}
\mbox{} \\[-24pt] 
\caption{
\label{fig:test_conjecture}
\textbf{Illustration of the structural stability ordering (\ref{eq:main-result}).}  We randomly generated $1000$ stable $3\times3$ matrices.  Entries were independently drawn from a normal distribution of zero mean and unit variance, discarding unstable matrices.  We see that the stochastic structural stability of a matrix, $\mathcal{S}_{\mathrm{_{STR}}}^{\mathrm{w}}$, can be much smaller than the smallest constant complex change needed to destabilize that matrix, $\mathcal{S}_{\mathrm{_{STR}}}^c$.  The equality 
$\mathcal{S}_{\mathrm{_{STR}}}^{\mathrm{w}}=2\mathcal{S}_{\mathrm{_{STrL}}}^c$ is attained for normal matrices.
}
\end{figure}

\section{Proving the theorem} 

To prove our main result, we follow a reasoning similar to the one that led to the identity (\ref{eq:formula_rC}).  Beyond establishing the validity of our claim, the proof shows how to explicitly construct destabilizing internal perturbation from external perturbations. A construction that could be useful for applications.

\noindent Let us start by showing that $\mathcal{S}_{\mathrm{_{STR}}}^{\mathrm{w}} \leq \mathcal{S}_{\mathrm{_{DYN}}}^{\mathrm{w}}$.  For a stable 
Jacobian matrix $A$, suppose that $\mathcal{S}_{\mathrm{_{DYN}}}^{\mathrm{w}}(A) = 1/v$.  By (\ref{eq:stochastic_variability}) this means that there exists two positive normalized matrices, the noise correlation matrix $\Sigma$ and the associated system response correlation matrix $\Pi$, such that 
\[
 \widehat{A}^{-1}\Sigma = v^{-1}\Pi
 \ \Leftrightarrow\ 
 \widehat{A}\Pi+v^{-1}\Sigma = 0.
\]
As in the deterministic setting, using $\Sigma$ and $\Pi$, we will construct a destabilizing operator $\mathcal{P}$.  However, for this operator to represent real internal noise, it must be of the form $\sum_{k}P_{k} \cdot P_{k}^{\top}$ for a set of real matrices $P_{k}$.  To construct such an operator, we use the spectral decomposition of the positive semi-definite matrices $\Sigma$ and $\Pi$,
\[
 \Sigma = \sum_{i=1}^n \lambda_{i}\boldsymbol{u}_{i}\boldsymbol{u}_{i}^{\top}
 \quad\text{and}\quad
 \Pi = \sum_{i=1}^n \mu_{i}\boldsymbol{v}_{i}\boldsymbol{v}_{i}^{\top},
\]
and put $P_{k} = \sqrt{\frac{\lambda_{i}\mu_{j}}{v}}\boldsymbol{u}_{i}\boldsymbol{v}_{j}^{\top}$, representing $n^2$ independent internal perturbations of the matrix $A$.  We have
\begin{align*}
 \mathcal{P}(C)
 &= \sum_{k}P_{k} C P_{k}^{\top} \\[-5pt]
 &= v^{-1}\sum_{i=1}^n \lambda_{i}\boldsymbol{u}_{i}\boldsymbol{u}_{i}^{\top}
    \sum_{j=1}^n \mu_{j}\langle \boldsymbol{v}_{j},C \boldsymbol{v}_{j} \rangle \\
 &= v^{-1} \mbox{Tr}\left(\Pi C\right)\Sigma,
\end{align*}
and, using the Hilbert-Schmidt inner product $\left\langle X,Y\right\rangle=\mbox{Tr}\left(X^{*}Y\right)$ from which the Frobenius norm derives,  we see that $\mathcal{P}$ takes the compact form
\[
 \mathcal{P} = v^{-1}\left\langle \Pi,\,\cdot\,\right\rangle \Sigma,
\]
showing that $\mathcal{P}(\Pi) = v^{-1}\Sigma$ and $\left\Vert\mathcal{P}\right\Vert = v^{-1}$.  We thus have that
\[
 \big(\widehat{A}+\mathcal{P}\big) \Pi=0.
\]
Hence, $\mathcal{P}$ destabilizes the lifted dynamics and corresponds to real internal noise of intensity $||\mathcal{P}|| = v^{-1}$.  Thus, 
\[
 \mathcal{S}_{\mathrm{_{STR}}}^{\mathrm{w}}(A)
 \leq \mathcal{S}_{\mathrm{_{DYN}}}^{\mathrm{w}}(A).
\] 
Conversely, suppose that $\mathcal{S}_{\mathrm{_{STR}}}^{\mathrm{w}}(A) = p$.  There exists an operator $\mathcal{P}$ with $\Vert\mathcal{P}\Vert = p$ such that $\widehat{A}+\mathcal{P}$ is unstable, i.e., it has a dominant eigenvalue on the imaginary axis.  There can be strictly imaginary dominant eigenvalues, but we show in appendix~\ref{sect:appA} that there is also a dominant eigenvalue at zero.  Hence, for some matrix $X$ with $\Vert X\Vert_{\mathrm{F}} = 1$,
\[
 (\widehat{A}+\mathcal{P}) X = 0
 \ \Leftrightarrow\ 
 X = -\widehat{A}^{-1}(Y),
\]
with $Y = \mathcal{P}(X)$.  Since $\Vert Y\Vert_{\mathrm{F}} \leq p$ we find that $\Vert\widehat{A}^{-1}\Vert \geq p$, so that by virtue of (\ref{eq:V_S}),
\[
 \mathcal{S}_{\mathrm{_{STR}}}^{\mathrm{w}}(A)
 \geq \mathcal{S}_{\mathrm{_{DYN}}}^{\mathrm{w}}(A),
\]
which concludes the proof of (\ref{eq:r_S_vs_V_S}).  We refer to appendix~\ref{sect:appB} for the proof of (\ref{eq:main-result}).

\section{Epilogue} 

\noindent We demonstrated that, at least locally, structural and dynamical stability are remarkably connected concepts, so that the dynamical response to external perturbations can inform on systems sensitivity to structural changes. Measures of dynamical stability were constructed to mimic empirical approaches to estimate stability from time-series data \cite{arnoldi_resilience_2016}, thus revealing a fundamental link between pragmatic empirical views on stability and the more abstract concept of structural stability.     

\noindent We quantified dynamical stability via the maximal system response to external perturbations, and structural stability via the minimal destabilizing internal perturbation. 
However, it is not necessary to consider these \emph{worst-case scenarios} for a connection between these two stability notions to hold. Indeed,  to any external perturbation and associated system response, corresponds a destabilizing internal perturbation. The larger the amplification of the external perturbation, the smaller the intensity of the corresponding destabilizing internal perturbation. 

\noindent We mentioned above that the dynamics defined by the SDE (\ref{eq:homogeneous_SDE}) corresponds to It\^o's interpretation of Wiener processes \cite{kampen_stochastic_1997}.  Such stochastic signals can be seen as trains of delta peaks, occurring at random instants \cite{arnoldi_resilience_2016}[Appendix B].  In It\^o's prescription, the intensity of random pulses should be determined by the state variables before the pulse.  For instance, the pulse $x(t)\delta(t-t_k)$ arriving at time $t_k$ should be multiplied by $x(t_k^-)$.  On the other hand, Statonovich's prescription would be to multiply that pulse by $(x(t_k^+)+x(t_k^-))/2$.  Obviously the two prescriptions yield the same SDE when the noise is additive, hence in our case, the same definition of dynamical stability $\mathcal{S}_{\mathrm{_{DYN}}}^{\mathrm{w}}$.  A difference occurs, the \emph{spurious drift}, when the noise is multiplicative $-$in our case, when it affects the Jacobian matrix as in (\ref{eq:homogeneous_SDE}).  This yields a different definition of stochastic structural stability then $\mathcal{S}_{\mathrm{_{STR}}}^{\mathrm{w}}$. We leave it as an open problem to transpose the relashionship between dynamical and structural stability, when using Stratonovich's interpretation.  Choosing between the two prescriptions depends on the physical origin of the noise \cite{kampen_stochastic_1997, smythe_ito_1983}.  If the system is intrinsically stochastic, then It\^o's interpretation should be used.  If the noise is created by the application of a random force on an otherwise deterministic system, then Stratonovich's interpretation makes more sense. 

\noindent It has long been acknowledged that local stability analysis is not sufficient to fully grasp the stability of attractors.  For instance, the size of basins of attraction can be a fundamental global feature, independent of local stability \cite{holling_resilience_1973, menck_how_2013}.  Yet, linear stability theory has served as a fundamental reference point to move beyond local stability analysis.  It also allows to qualitatively depict the dynamical picture in the vicinity of equilibriums, without further knowledge of the global field.  The same can be said for structural stability.  In this paper we showed that the methodology used to quantify local dynamical stability also provides a measure of local structural stability.  This ought to serve as a benchmark to improve global, quantitative analysis of structural stability. 





\section*{acknowledgements}

For numerous discussions related to this article, the authors would like to thank Michel Loreau and Jos\'e Montoya from the CBTM and the  Ecological Networks and Global Change group. This work is supported by the TULIP Laboratory of Excellence (ANR-10-LABX-41), the AnaEE-France project (ANR-11-INBS-0001) and by the BIOSTASES Advanced Grant, funded by the European research council under the European Union's Horizon 2020 research and innovation programme  (grant agreement No 666971).



\bibliographystyle{ieeetr} 
\bibliography{bib_rspa.bib}

\newpage
\appendix

\section{Dominant eigenvalues of perturbed lifted operator} 
\label{sect:appA}

Recall that $\widehat{A}$ is the lifted operator defined from a stable real matrix $A$ and acting on any matrix $X$ as $\widehat{A}(X)=AX+XA^\top$, and that $\mathcal{P}$ is defined from an arbitrary sequence of real matrices $P_k$ as $\mathcal{P}(X)=\sum_k P_k X P_k^\top$.  Assume that the perturbed operator $\widehat{A}+\mathcal{P}$ lies on the boundary between stability and instability, that is,
\[
\alpha(\widehat{A}+\mathcal{P})=0.
\]
Here we show that any operator of the form 
\[
 \mathcal{A}_{\epsilon}=\widehat{A}+\epsilon \mathcal{P}
 \quad \text{with }\ 0\leq\epsilon<1,
\]
must have a real dominant eigenvalue $\lambda_\epsilon<0$, associated to an eigenvector $X_\epsilon$.  This would show, in particular, that $\lambda_\epsilon \rightarrow 0$ as $\epsilon$ goes to $1$, so that 
\[
 \mathcal{A}_{\epsilon}(X_\epsilon) \rightarrow (\widehat{A}+\mathcal{P}) X_1 = 0,
\]
an identity that was previously needed to prove that 
$\mathcal{S}_{\mathrm{_{STR}}}^{\mathrm{w}}(A)
\geq \mathcal{S}_{\mathrm{_{DYN}}}^{\mathrm{w}}(A)$. 

\noindent To show this, suppose the converse, that is: that the dominant eigenvalues of $\mathcal{A}_{\epsilon}$ all have non-zero imaginary parts.  An arbitrarily small perturbation can ensure that $\mathcal{A}_{\epsilon}$ has a unique dominant eigenvalue $i\omega_\epsilon+\alpha_\epsilon$ up to complex conjugacy, associated to an eigenvector $X_{\epsilon}$.  Suppose now that 
\[
 \left\langle \mathbb{I}, X_{\epsilon}\right\rangle= \text{Tr}X_\epsilon \neq 0, 
\]
which again can be insured by arbitrarily small perturbations.  The flow 
\[
 \dot{X}=\mathcal{A}_\epsilon(X),
\]
preserves the set of real positive matrices.  In particular, the starting point $C_{0}=\mathbb{I}$ becomes, as time flows forward: 
\[
 C_{t} = e^{\alpha_\epsilon t} \left\{ e^{i\omega_\epsilon t}
 \left\langle X_{\epsilon},\mathbb{I}\right\rangle X_{\epsilon}
 +\mbox{c.c.}+o(1)\right\}.
\]
Writing $Z_\epsilon=\left\langle X_{\epsilon},\mathbb{I}\right\rangle X_{\epsilon}$, we see that $C_t$ converges to
\[
 e^{\alpha_\epsilon t} \left\{\cos\left(\omega_\epsilon t\right)
 \Re(Z_\epsilon)-\sin(\omega_\epsilon t)\Im(Z_\epsilon)\right\},
\]
which rotates at frequency $\omega_{\epsilon}$.  It therefore cannot be positive for all $t$ which it should when the subdominant terms in $C_{t}$ become negligible.  We thus get a contradiction, hence $\lambda_\epsilon$ must be real. 

\noindent To summarize, we have shown that, modulo arbitrary small perturbations, the dominant eigenvalue of $\mathcal{A}_\epsilon$ is simple and real.  Because the spectrum depends continuously on the matrix entries \cite{kato_perturbation_1995}, this implies that amongst the dominant eigenvalues of $\mathcal{A}_\epsilon$ one was already real.

\newpage
\section{Ordering of structural stability measures} 
\label{sect:appB}

In the theorem  we claim that stochastic structural stability relates to the stability radius of matrices following the general inequality (illustrated in figure~\ref{fig:test_conjecture}):
\[
 \mathcal{S}_{\mathrm{_{STR}}}^{\mathrm{w}}
 \leq 2 \mathcal{S}_{\mathrm{_{STR}}}^c,
\]
with equality when the Jacobian matrix at the equilibrium is normal.  Here we prove this fact. Let us start by stating a lemma from linear algebra.

\begin{lemma}
For any invertible matrix $B$ acting on $\mathbb{C}^n$, it holds that:
\[
 \min_{x \in\mathbb{C}^n ;\; \Vert x \Vert = 1}
 \left\Vert B x \right\Vert
 = \Big( \max_{y \in\mathbb{C}^n ;\; \Vert y \Vert = 1}
 \big\Vert B^{-1} y \big\Vert \Big)^{-1}.
\]
\end{lemma}

\begin{proof}
Take $x_*=B^{-1}y/\Vert B^{-1}y \Vert$ with $y$ normalized and realizing the maximum of $\Vert B^{-1} y \Vert$.  By construction, 
\[
 \min_{x \in\mathbb{C}^n ;\; \Vert x \Vert = 1} \Vert B x \Vert
 \leq \left\Vert B x_* \right\Vert 
 = \Big( \max_{y \in\mathbb{C}^n ;\; \Vert y \Vert = 1}
 \big\Vert B^{-1} y \big\Vert \Big)^{-1}.
\]
To show that taking the minimum over all normalized elements $x$ does not give anything smaller, it suffices to choose $y_*=Bx/\Vert Bx \Vert$ with $x$ normalized and realizing the minimum of $\Vert Bx \Vert$.  By construction,
\[
 \max_{y \in\mathbb{C}^n ;\; \Vert y \Vert = 1} \big\Vert B^{-1} y \big\Vert
 \geq \big\Vert B^{-1} y_* \big\Vert 
 = \Big( \min_{x \in\mathbb{C}^n ;\; \Vert x \Vert = 1} \Vert B x \Vert 
 \Big)^{-1},
\]
which is equivalent to
\[
 \min_{x \in\mathbb{C}^n ;\; \Vert x \Vert = 1} \Vert B x \Vert
 \geq \Big( \max_{y \in\mathbb{C}^n ;\; \Vert y \Vert = 1} \big\Vert B^{-1} y  
 \big\Vert \Big)^{-1},
\]
proving the lemma.
\end{proof}

\noindent With this result in hand, we can write, for any stable real matrix $A$,  
\[
 \mathcal{S}_{\mathrm{_{STR}}}^{\mathrm{w}}(A)
 = \big\Vert \widehat{A}^{-1} \big\Vert^{-1}
 = \min_{\Vert X \Vert _{\mathrm{F}}=1} \big\Vert \widehat{A}X 
 \big\Vert_{\mathrm{F}}.
\]
In particular, for any normalized matrix $X$, 
\[
 \mathcal{S}_{\mathrm{_{STR}}}^{\mathrm{w}}(A)
 \leq \big\Vert \widehat{A}X \big\Vert_{\text{F}}
 = \big\Vert AX+XA^{\top} \big\Vert_{\mathrm{F}}.
\]
Choosing $X$ as a rank-one orthonormal projector, $X=\boldsymbol{w}\boldsymbol{w}^{*}$, gives, for any real $\omega$:
\begin{align*}
 \mathcal{S}_{\mathrm{_{STR}}}^{\mathrm{w}}(A)
 &\leq \big\Vert \left(A\boldsymbol{w}\right)\boldsymbol{w}^{*}
 +\boldsymbol{w}\left(A\boldsymbol{w}\right)^{*} \big\Vert_{\mathrm{F}} \\
 &\leq \big\Vert 
 \left(\left(i\omega-A\right)\boldsymbol{w}\right)\boldsymbol{w}^{*}
 +\boldsymbol{w}\left(\left(i\omega-A\right)\boldsymbol{w}\right)^{*} 
 \big\Vert_{\mathrm{F}}.
\end{align*}
On the other hand we also have, using the above lemma, that
\[
 \mathcal{S}_{\mathrm{_{STR}}}^c(A)
 = \big\Vert (i\omega -A)^{-1} \big\Vert^{-1}
 = \inf_{\omega, ||\boldsymbol{w}|| = 1} \Vert (i\omega-A) \boldsymbol{w} \Vert.
\]
In the upper bound of $\mathcal{S}_{\mathrm{_{STrL}}}^{\mathrm{w}\mathfrak{noise}}(A)$, choosing $\omega$ to be the system's resonant frequency and $\boldsymbol{w}$ the associated minimizing vector of $\Vert (i\omega-A) \boldsymbol{w} \Vert$, and then invoking the triangular inequality, yields
\[
 \mathcal{S}_{\mathrm{_{STR}}}^{\mathrm{w}}(A)
 \leq 2 \mathcal{S}_{\mathrm{_{STrL}}}^c(A).
\]

\noindent Let us now show that equality holds whenever $A$ is normal.  First of all, for normal $A$, $\mathcal{S}_{\mathrm{_{STrL}}}^{\mathfrak{cst}}(A)$ coincides with the spectral abscissa $\alpha\left(A\right)$.  This a consequence of the following equality, valid for any normal matrix $A$ and complex number $z$ away from the spectrum of $A$ \cite{trefethen_spectra_2005}: 
\begin{equation}
 \big\Vert (z-A)^{-1} \big\Vert = 
 1/ \mbox{dist}(z,\mbox{spect}(A)).\label{Hausdorff_dist}
\end{equation}

where $\mbox{dist}(\cdot,\cdot)$ stands for the Hausdorff distance between subsets of the complex plane, equipped with the Euclidean metric. Indeed, taking $z=i\omega$, where $\omega$ is the imaginary part of the dominant eigenvalue of $A$,  gives $\mathcal{S}_{\mathrm{_{STrL}}}^{\mathfrak{cst}}(A)=\alpha(A)$.  Also, if $A$ is normal, $\widehat{A}$ is automatically normal as well.  Since $A$ is diagonalizable, we can express the spectrum of $\widehat{A}$ from the one of $A$.  Indeed, if $\{ (\lambda_{i},\boldsymbol{u}_{i}) \}_{i}$ are the complete eigenpairs of $A$, then
$\big\{ (\lambda_{i}+\bar{\lambda}_{j},\,\boldsymbol{u}_{i}\boldsymbol{u}_{j}^{*}) \big\}_{i,j}$
are the complete eigenpairs of $\widehat{A}$.  If $\lambda_{0}$ is the dominant eigenvalue of $A$, then by definition $-\Re\left(\lambda_{0}\right)=\alpha(A)$, and thus
$\big\{ -2\alpha(A), 2\lambda_{0}, 2\bar{\lambda}_{0} \big\}$
are dominant eigenvalues of $\widehat{A}$. Applying the above identity (\ref{Hausdorff_dist}) to the normal operator $\widehat{A}$, namely:
\[
 \big\Vert (z-\widehat{A})^{-1} \big\Vert = 
 1/ \mbox{dist}(z,\mbox{spect}(\widehat{A})),
\]
and taking $z=0$ gives $\Vert \widehat{A}^{-1} \Vert= 1/2\alpha(A)$, hence $\mathcal{S}_{\mathrm{_{STR}}}^{\mathrm{w}}(A) = 2 \alpha(A)= 2\mathcal{S}_{\mathrm{_{STR}}}^c(A)$, which is the expected equality.

\noindent Finally, since the real constant perturbation $P=\alpha(A)\mathbb{I}$ always sufficient to destabilize any stable matrix $A$, in light of the previous result, we see that for normal matrices 
\[
\mathcal{S}_{\mathrm{_{STR}}}^{\Re(c)}(A)=\mathcal{S}_{\mathrm{_{STR}}}^c(A)
\]
completing the the proof of the theorem.

\end{document}